\documentclass[a4paper,reqno,12pt]{amsart}
\usepackage[english]{babel}
\usepackage{amsmath}
\usepackage[colorlinks=true,linkcolor=blue,urlcolor=blue,citecolor=red]{hyperref}
\usepackage{xcolor}
\usepackage[sort]{cite}
\usepackage{amssymb}
\usepackage{amscd}
\usepackage{amsthm}
\usepackage{euscript}
\usepackage{tikz}
\newtheorem{proposition}{Proposition}
\newtheorem{lemma}{Lemma}
\newtheorem{theorem}{Theorem}

\newtheorem{corollary}{Corollary}
\theoremstyle{definition}
\newtheorem{definition}{Definition}

\theoremstyle{remark}
\newtheorem {remark}{Remark}

\def\BG{{\mathbb G}}
\def\BK{{\mathbb K}}

\def\BG{{\mathbb G}}

\def\BK{{\mathbb K}}

\def\BN{{\mathbb N}}

\def\BP{{\mathbb P}}
\def\BA{{\mathbb A}}

\def\Soc{\mathrm{Soc}}

\def\mf{\mathfrak{m}}
\def\Ann{\mathrm{Ann}}

\title{Additive actions on projective hypersurfaces\\ with a~finite number of orbits}

\thanks{The work was supported by the grant RSF \texttt{23-71-01100}}

\author{Viktoriia Borovik}
\email{vborovik@uni-osnabrueck.de}
\address{Osnabr\"uck University,
Fachbereich Mathematik/Informatik
Albrechtstr.~28a,
49076 Osnabrück, Germany.}

\author{Alexander Chernov}
\email{chenov2004@gmail.com}
\address{Lomonosov Moscow State University, Faculty of Mechanics $\&$ Mathematics, Department of Higher Algebra, Leninskie Gory~1, 119991 Moscow, Russia; \newline
$\&$ Higher School of Economics, Faculty of Computer Science, Pokrovsky Boulvard 11, 109028 Moscow, Russia.}
\author{Anton Shafarevich}
\email{shafarevich.a@gmail.com}
\address{
Moscow Center for Fundamental and Applied Mathematics, Moscow, Russia; \newline
$\&$
Lomonosov Moscow State University, Faculty of Mechanics $\&$ Mathematics, Department of Higher Algebra, Leninskie Gory~1, 119991 Moscow, Russia; \newline
$\&$ Higher School of Economics, Faculty of Computer Science, Pokrovsky Boulvard 11, 109028 Moscow, Russia.}

\subjclass[2020]{Primary 14L30, 14J70; Secondary 13E10.}
\keywords{Algebraic variety, algebraic group, unipotent group, algebraic group actions, open orbit, additive action, local algebra, projective hypersurface}
\begin{document}
\maketitle

\begin{abstract}
An induced additive action on a projective variety $X \subseteq \mathbb{P}^n$ is a regular action of the group $\mathbb{G}_a^m$ on $X$ with an open orbit, which can be extended to a regular action on the ambient projective space~$\mathbb{P}^n$. In this work, we classify all projective hypersurfaces admitting an induced additive action with a finite number of orbits.

\end{abstract}

\section{Introduction}

Let $\BK$ be an algebraically closed field of zero characteristic. By a variety or an algebraic group we always mean an algebraic variety or an algebraic group over $\BK$. By open and closed subsets of algebraic varieties we always mean open and closed subsets in Zariski topology. We denote by $\BG_a = (\BK, +)$ the additive group of the ground field and by $\BG_a^m$ the group $$\BG_a^m = \underbrace{\BG_a \times \cdots \times \BG_a}_\text{$m$ times}.$$
\begin{definition}
An \emph{additive action} on an algebraic variety $X$ is a regular effective action of~$\BG_a^m$ on $X$ with an open orbit. By an \emph{induced additive action} on an embedded projective algebraic variety $X \subseteq \BP^n$ we mean a regular effective action of $\BG_a^m$ on $\BP^n$ such that the variety $X$ is the~closure of an orbit of~$\BG_a^m$.
\end{definition}
Not every additive action on a projective variety is induced. An example can be found in~\cite[Example 1]{AP}. However, when the projective variety $X\subseteq \BP^n$ is normal and linearly normal, then every additive action of $\BG_a^m$ on $X$ lifts to the regular effective action of $\BG_a^m$ on the projective space $\BP^n$.

In \cite{HT} a remarkable correspondence between additive actions on the projective space~$\BP^n$ and local algebras of dimension $n+1$ was obtained. By a \emph{local algebra} we mean a commutative associative algebra over $\BK$ with a unit and a unique maximal ideal. We will recall this correspondence in Section~\ref{prem}. A correspondence between actions of arbitrary commutative algebraic groups on $\BP^n$ with an open orbit and associative commutative algebras with a unit element of dimension $n+1$ was established in \cite{KL}.

 The systematic study of additive actions on projective and complete varieties was initiated in \cite{Sh, AS, AP}. There are several results on additive actions on projective hypersurfaces. For example, it was proven in \cite{Sh} that  there is a unique additive action on a non-degenerate quadric. This result was generalized in \cite{BG}, where actions of arbitrary algebraic commutative groups on non-degenerate quadrics with an open orbit were described. In  \cite{AP} and \cite{AS} induced actions on projective hypersurfaces were studied. It was proven in \cite{AZ} that a non-degenerate hypersurface (see Definition \ref{nond}) admits at most one additive action. When a degenerate hypersurface admits an additive action, then there are at least two non-isomorphic additive actions on it, see \cite{Be}. For additive actions on degenerate hypersurfaces we refer  also to \cite{Li}.

Flag varieties admitting an additive action were classified in \cite{A} and all additive actions on flag varieties were classified in \cite{De}. Additive actions on toric varieties were studied in \cite{AR, APS, Ds1, Ds2, Sk, Sha, Sha2}. There are results on additive actions on Fano varieties in \cite{DL, DL2, BH, BM, ZM}. For a detailed review of the results on additive actions we refer to \cite{AZ}.

Among actions of algebraic groups on algebraic varieties, actions with a finite number of orbits are of particular interest. For example, toric varieties can be characterized as varieties on which an algebraic torus acts with a finite number of orbits. Spherical varieties admit an action of a reductive group with a finite number of orbits. Additive actions with a finite number of orbits on complete varieties, with an additional condition on the actions of one-dimensional subgroups, were described in the work \cite{CC}; see also Section \ref{LP}.

In this paper we find all projective hypersurfaces admitting an induced additive action with a~finite number of orbits. We use the technique developed in \cite{AP, AS, AZ, Sh}, generalizing the correspondence from \cite{HT, KL}. Each hypersurface with an induced additive action corresponds to a pair $(A,U)$, where $A$ is a local algebra with the maximal ideal $\mf$ and $U$ is a subspace in $\mf$ of codimension 1 generating $A$ as an algebra with a unit. We classify all such pairs $(A, U)$ that correspond to hypersurfaces admitting an~induced additive action with a~finite number of orbits, see Theorem \ref{maintheor}. By a pair $(A,U)$, one can find an equation defining the hypersurface using \cite[Theorem 2.14]{AZ}.

Our final results are stated in Theorem \ref{maintheor} and Corollary \ref{finalcor}. Geometrically they mean the following.

\begin{itemize}
    
    \item[a) ] There is exactly one curve in $\BP^2$ which admits an induced additive action with a finite number of orbits. It is $\BP^1$ embedded in $\BP^2$ via Veronese embedding of degree 2.
    \item[b) ] There are exactly three surfaces in $\BP^3$ which admit an induced additive action with a finite number of orbits. They are 
    \begin{enumerate}
        \item $\BP^1\times \BP^1$ embedded in $\BP^3$ via Segre embedding;
        \item the non-degenerate cubic;
        \item the degenerate hypersurface of degree 2 which is the projective cone over the hypersurface from the point a).
    \end{enumerate}
    \item[c) ] When $n>3$, there are exactly two hypersurfaces in $\BP^n$ which admit an induced additive action with a finite number of orbits. One of them is a non-degenerate hypersurface $X_n$ of degree $n$. The other one is a degenerate hypersurface $Y_{n-1}$ of degree $n-1$ which is the projective cone over $X_{n-1}.$

\end{itemize}

The structure of the text is as follows. In Section \ref{prem} we recall known facts about additive actions on projective space and projective hypersurfaces. In Section \ref{MR} we prove the main results and find all projective hypersurfaces that admit an additive action with a finite number of orbits. Finally, in Section \ref{PROP} we discuss properties of obtained hypersurfaces such as the structure of orbits, smoothness, normality and the total number of different additive actions.

\section*{Acknowledgments.} The authors are grateful to the anonymous reviewer for careful reading and valuable suggestions for improving the text.

\section{Additive actions on projective varieties}\label{prem}

In this section we recall some  facts on additive actions on projective varieties. We say that two induced additive actions on a projective variety $X \subseteq \BP^n$ are \emph{equivalent} if one is obtained from the other via an automorphism of $\BP^n$ preserving $X$.

\begin{proposition}\cite[Proposition 2.15]{HT}\label{hassett}\label{hassett}
There is a one-to-one correspondence between

\begin{enumerate}

     \item equivalence classes of additive actions on $\BP^n$; 
    \item isomorphism classes of local algebras of dimension $n+1$.
\end{enumerate}
\end{proposition}

We now recall how to construct an additive action on~$\BP^n$ by an $(n+1)$-dimensional local algebra $A$. Let $\mf$ be the maximal ideal in $A$. Then $A = \BK \oplus \mf$ (a direct sum of  vector spaces) and all elements in the ideal $\mf$ are nilpotent. This is a well-known fact, for the proof we refer to \cite[Lemma 1.2]{AZ}.
Consider the exponential map on $\mf$:
$$m \mapsto \exp(m) = \sum_{i\geq 0} \frac{m^i}{i!}, \; \text{ for }m \in  \mf.$$
This map is well-defined on $\mf$. The additive group of $\mf$ is isomorphic to $\BG_a^n$ and $\mf$ acts on the algebra~$A$ by the following rule:
$m \circ a = \exp(m) \cdot a$.
This is an algebraic action. The stabilizer of a unit is trivial, so we have the following isomorphisms of algebraic varieties
$$\BA^n \simeq\BG_a^n \simeq \mathrm{exp}(\mf) \cdot 1 = 1 + \mf,$$
where the last equality is satisfied since the map
$$1+m \mapsto \ln(1+m) = \sum_{i>0} (-1)^{i-1}\frac{m^i}{i}, \; \text{ for }m \in  \mf,$$
is well-defined on $1+\mf$ and $\exp(\ln(1+m)) = 1+ m$.
The action of $\mf$ on the algebra $A$ defines an algebraic action of $\BG_a^n$ on the projective space $  \BP^{n} = \BP(A)$ by the rule
$$m \circ p(a) = p(\exp(m) \cdot a),$$
where the map $p\colon A\setminus \{0\} \to \BP(A)$ is the canonical projection. The orbit of $p(1)$ is the open orbit, so this defines an additive action on $\BP^n$. See \cite[Example 1.50]{AZ} for further examples of this~construction.

\begin{proposition}\cite[Proposition 3]{AP}\label{APprop}
There is a one-to-one correspondence between
\begin{enumerate}\label{prophyp}
    \item equivalence classes of pairs $(X, \alpha),$ where $X$ is a hypersurface in $\mathbb{P}^n$ and $\alpha$ is an induced additive action on $X$; 
    \item isomorphism classes of pairs $(A, U)$, where $A$ is a local $(n+1)$-dimensional algebra with the maximal ideal $\mathfrak{m}$ and $U$ is an $(n-1)$-dimensional subspace in $\mathfrak{m}$ that generates $A$ as an algebra with a unit. 
\end{enumerate}
\end{proposition}
The pairs $(A,U)$ from Proposition \ref{prophyp} are called  \emph{$H$-pairs}. We say that two $H$-pairs $(A_1, U_1)$  and $(A_2, U_2)$ are isomorphic if there is an isomorphism of local algebras $\varphi\colon A_1 \to A_2$ such that $\varphi(U_1) = U_2.$

Now we fix an $H$-pair $(A,U)$ until the end of the section. Let $\mf$ be the maximal ideal of~$A$. In Proposition \ref{prophyp}, the additive action on $X$ is defined as follows. We define the action of $\mf$ on $\BP(A)$ in the same way as in Proposition \ref{hassett}. Thus we restrict this action to the subgroup $U\simeq \BG_a^{n-1}$ and consider the subvariety
$$X = p(\overline{\exp(U) \cdot 1}).$$
Then $X$ is a hypersurface in $\BP(A) = \BP^n$ and the group $U$ acts on $X$ with an open orbit.

The following results illustrate how to find the defining equation and degree of $X$.

\begin{theorem}\cite[Theorem 5.1]{AS}\label{degree} The degree of the hypersurface $X$ is equal to the largest number $d \in \mathbb{N}$ such that $\mf^d \nsubseteq U$, where $\mf$ is the maximal ideal in the corresponding local algebra $A$. 
\end{theorem}

\begin{theorem}\cite[Theorem 2.14]{AZ}\label{equation}
    The hypersurface $X$ is given in $\BP(A)$ by the following homogeneous equation:
    $$z_0^d\pi\left(\ln\left(1 + \frac{z}{z_0} \right)\right) = 0,$$
    where $z_0\in \BK, \, z\in \mf$ and $\pi\colon \mf \to \mf/U\simeq \BK$ is the canonical projection.
\end{theorem}

It is also possible to describe elements $a\in A$ such that $\pi(a) \in X.$

\begin{proposition}\cite[Corollary 2.18]{AZ}\label{complement}
The complement of the open $U$-orbit in $X$ is the set
$$\{ \, p(z)  \mid z \in \mf \text{ such that } z^d \in U \, \},$$
where $p\colon A\setminus\{0\} \to \BP(A)$ is the canonical projection and $d$ is the degree of $X$.
\end{proposition}

\begin{corollary}\label{morb}
Suppose that the point $x \in X$ belongs to the complement of the open orbit of the group~$U$. Then the $\mf$-orbit of $x$ is contained in $X$.
\end{corollary}
\begin{proof}
    Let us take $z \in \mf$ such that $p(z) = x$ lies in $X$. Then $z^d \in \mf^d\cap U$. The $\mf$-orbit of the element $z$ is $z+z \cdot \mf$. But then $(z+z \cdot \mf)^d \subseteq z^d + \mf^{d+1}\subseteq U$. So $p(z+z \cdot \mf) \subseteq X$.
\end{proof}

We recall that a \emph{socle} of a local algebra $A$ is the ideal $\mathrm{Soc}(A) := \{z\in A \mid z\cdot \mf = 0\}$.

\begin{corollary}
    The set $\{ \, p(z) \mid z \in \Soc(A)\setminus\{0\} \, \}$ is contained in $X$.
 \end{corollary}
 \begin{proof}
    For  all $z\in \Soc(A)$ we have $z^d = 0$ is in the group $U$.
 \end{proof}

\begin{corollary}\label{Soclem}
    If  $\mathrm{dim} (\Soc(A)) > 1$ then there are infinitely many $U$-orbits on $X$.
\end{corollary}
\begin{proof}
    If $z\in \Soc(A)$ then $\exp(U) \cdot z = \{z\}$. So the set $\{ \, p(z) \mid z \in \Soc(A)\setminus\{0\} \, \} \subseteq X$ consists of the $U$-fixed points and has dimension at least 1.
\end{proof}

It is also possible to describe the relationship between $\mf$-orbits and $U$-orbits on $X$ . For an element $z\in A$ we denote by $\mathrm{Ann}(z)$ the ideal $\{a\in A \mid az = 0\}.$

\begin{proposition}\label{Uorbits}
    Let $z \in \mf\setminus\{0\}$ be an element with $p(z) \in X$. 
    \begin{enumerate}
        \item If $\mathrm{Ann}(z)+U = \mf$, then the $\mf$-orbit of $p(z)$ coincides with the $U$-orbit. 
        \item Otherwise, $\Ann(z) \subseteq U$ and the $\mf$-orbit of $p(z)$ is the union of an infinite number of~$U$-orbits.
     \end{enumerate}
\end{proposition}
\begin{proof}

    We will show that $\Ann(z)$  coincides with the stabilizer $\mathrm{St}_{\mf}(p(z))$ with respect to the~$\mf$-action. The inclusion $\Ann(z) \subseteq \mathrm{St}_{\mf}(p(z))$ is clear. Indeed, if $a\in \Ann(z)$ then we have $az = 0$ and $\exp(a) \cdot z = z$. We now show the reverse inclusion.
    If $a \in \mathrm{St}_{\mf}(p(z))$ then 
    $$
        \exp(a) \cdot p(z) = p(z + az+ \frac{a^2}{2}z + \frac{a^3}{6}z +\ldots) = p(z),
    $$
    which implies
    $$
    z + az+ \frac{a^2}{2}z + \frac{a^3}{6}z + \ldots = \lambda z.
    $$
    for some $\lambda \in \BK.$
    There is a number $k \in \mathbb{N}$ such that $z\in \mf^{k}\setminus \mf^{k+1}$. Then the element $az + \frac{a^2}{2}z +  \frac{a^3}{6}z+ \ldots \ $ lies in the ideal $\mf^{k+1}$. So we have 
    $$az + \frac{a^2}{2}z +  \frac{a^3}{6}z+ \ldots \ = (\lambda - 1)z,$$
    which is possible only when $\lambda = 1.$ Therefore,
    $$az + \frac{a^2}{2}z +  \frac{a^3}{6}z+ \ldots \ = 0.$$
    If $az \neq 0$ then there is $r$ such that $az\in \mf^{r}\setminus \mf^{r+1}.$ But  $\frac{a^2}{2}z +  \frac{a^3}{6}z+ \ldots \ \in \mf^{r+1}.$ So $az = 0$ and $a\in \Ann(z).$

    Thus, the $\mf$-orbit of $p(z)$ is isomorphic to $\mf/\Ann(z)$ and $U$-orbit of $p(z)$ is isomorphic~to 
    $$U/(\Ann(z)\cap U)\simeq (U+\Ann(z))/\Ann(z).$$
    Hence, if $U+ \Ann(z) = \mf$ then the $\mf$-orbit of $p(z)$ coincides with the $U$-orbit, and if $U+\Ann(z) \neq \mf$ then the action of $U$ on $\mf/\Ann(z)$ has infinitely many orbits. Since the codimension of $U$ in $\mf$ is 1, in the last case we have $\Ann(z) \subseteq U.$
\end{proof}

\section{Main result}\label{MR}
In this section we state our main result.
Recall that for a local algebra $A$ with the maximal ideal $\mf$ the following sequence of numbers 
$$(\dim_{\BK} A/\mf, \ \dim_{\BK} \mf/\mf^2, \ \dim_{\BK} \mf^2/\mf^3, \ \ldots \ ) $$
is called a \emph{Hilbert-Samuel sequence}.  

\begin{proposition}
    Let $(A, U)$ be an $H$-pair and $X$ be the corresponding hypersurface in~$\BP(A).$ Suppose that there are finitely many $U$-orbits in $X$. Then the Hilbert-Samuel sequence of~$A$ is either $(1,  1,1,\ldots, 1)$ or $(1, 2, 1, \ldots 1)$.
    
\end{proposition}
\begin{proof}
    Since $\BK$ is algebraically closed and $A/\mf$ is a finite-dimensional field over $\BK$  we have $$\dim_{\BK} A/\mf = 1.$$ 
    Suppose there is a number $k\geq 2$ with $\dim_{\BK} \mf^k/\mf^{k+1} > 1.$ For all $z\in  \mf^k\setminus\{0\}$ we have 
    $$z^d \in \mf^{kd} \subseteq \mf^{d+1}\subseteq U,$$ where $d$ is the degree of $X$. 
    Then $p(z)$ lies in $X$ for all $z \in \mf^k \setminus \{0\}$. The $\mf$-orbit of $p(z) $ is $p(z+z \cdot \mf)\subseteq p(z+\mf^{k+1})$. Thus, if the images of elements $z_1$ and $z_2$ from $\mf^k$ in $\mf^k/\mf^{k+1}$ are not proportional, then the $\mf$-orbits of $p(z_1)$ and $p(z_2)$ do not coincide, so their $U$-orbits are also different. Therefore, if $\dim_{\BK} \mf^k/\mf^{k+1} > 1$, there are infinitely many $U$-orbits on $X$, this contradicts our assumption.

    It implies that the Hilbert-Samuel sequence has the form $(1, r, 1, \ldots, 1)$. Now suppose that $r\geq 3$ and consider the map:
    \begin{align*}
    \varphi\colon \mf/\mf^2 &\to \mf^{d}/\mf^{d+1},\\
    z + \mf^2 &\mapsto z^d + \mf^{d+1}.
    \end{align*}
    The map $\varphi$ is a morphism between algebraic varieties $\mf/\mf^2 \simeq \BA^r$ and $\mf^d/\mf^{d+1} \simeq \BA^1$. The set  $Z : = \varphi^{-1}(0+\mf^{d+1})$ is non-empty, so $\mathrm{dim}(Z) \geq r-1 \geq 2$. For all elements $z\in \mf\setminus\{0\}$ with $z+\mf^2 \in Z$ we have that $p(z)$ is lying  in $X$. As previously, when elements $z_1 + \mf^2 \in Z$ and $z_2+\mf^2 \in Z$ are not proportional then the $U$-orbits of $p(z_1)$ and $p(z_2)$ are different. 
    
    Since $\dim (Z) \geq 2$ there are infinitely many $U$-orbits on $X$. This contradicts our assumption, thus $r\leq 2$ and the Hilbert-Samuel sequence of~$A$ equals $(1,  \ldots, 1)$ or $(1, 2, 1, \ldots 1)$.
\end{proof}

\begin{proposition}\label{algebras}
    
    Let $(A, U)$ be an $H$-pair and $X$ be the corresponding hypersurface in~$\BP(A).$ Suppose that there are finitely many $U$-orbits in $X$. Then for $n \geq  1$ we have 
    $$A \simeq \BK[x]/(x^{n+1}) \quad \text{or} \quad A \simeq \BK[x,y]/(xy, x^{3}, y^2 - x^{2}).$$
\end{proposition}
\begin{proof}
    First suppose that the Hilbert-Samuel sequence of $A$ equal to $(1, 1, \ldots, 1).$ Then $A$ is generated by one nilpotent element, so $A$ is isomorphic to $\BK[x]/(x^{n+1}).$

    Now consider the case when the Hilbert-Samuel sequence of $A$ is $(1, 2, \ldots, 1).$ Denote by~$r$ the maximal number such that $\mf^r \neq 0$, where $\mf$ is the maximal ideal in $A$. If $r = 1$ then $A\simeq \BK[x,y]/(x^2, xy, y^2)$. In this case $\Soc(A) = \langle x, y\rangle$, this contradicts Corollary \ref{Soclem}.

    Now consider the case $r>1$. Then there is an element $x \in \mf$ such that $\langle x^r \rangle = \mf^r$, see \cite[Lemma 2.13]{AZ}. Hence, $\mf = \langle x, x^2, \ldots, x^r, y \rangle$ where $y\in \mf\setminus \mf^2$ and images of $x$ and~$y$ are linearly independent in $\mf/\mf^2$. Thus, $xy$ lies  in $\mf^2$, so $xy = f(x)$, where $f(x)$ is a polynomial divisible by $x^2$. We replace $y$ with $y-\frac{f(x)}{x}$ to obtain $xy = 0$. 

    The element $y^2$ belongs to $\mf^2.$ Thus, $y^2 = g(x)$, where $g(x)$ is a polynomial divisible by $x^2.$ Assume that $y^2 = 0$, then $\Soc(A) = \langle x^r, y\rangle$, which contradicts Corollary \ref{Soclem}. On the other hand, $xy^2 = (xy)y = 0 = xg(x).$ It implies that $g(x) = \lambda x^r$, where $\lambda \in \BK\setminus \{0\}.$ We replace~$y$ with $\sqrt{\lambda}y$ to get $y^2 = x^r$. Then $A$ is isomorphic to the algebra $\BK[x,y]/(xy, x^{r+1}, y^2 - x^{r})$.

To complete the proof we should show that $r\leq 2$. Assume the converse, i.e., $r>2$.
    If we denote by $d \geq 2$ the degree of the hypersurface $X$, then 
    $\mf^{d+1} \subseteq U$. We have
    $$(y+\alpha x^2)^d = y^d + \alpha^dx^{2d} \in \mf^{d+1} \quad \text{for all }\alpha \in \BK.$$
Here we use that $y^2 = x^r$ and $ y^3 = 0$. Therefore, $p(y+\alpha x^2)\in X$ for all $\alpha \in \BK$. The $\mf$-orbit of $(y+\alpha x^2)$ is the set 
$$y + \alpha x^2 + (y + \alpha x^2)\mf \subseteq y + \alpha x^2 + \mf^3.$$ That is, the $\mf$-orbits of the points $p(y + \alpha x^2)$ do not coincide for different $\alpha$. Hence, if $r > 2$ there are infinitely many $U$-orbits on $X$, which leads to a contradiction.
 \end{proof}

 \begin{remark}\label{rem}
     Note that the algebra $\BK[x,y]/(xy, x^3, y^2 - x^2)$ is isomorphic to~$\BK[x,y]/(x^2, y^2)$. To see this, one should  take $\tilde{x} = y- ix$ and $\tilde{y} = y + ix$, then we get 
     $$\BK[x,y]/(xy, x^3, y^2 - x^2) = \BK[\tilde{x},\tilde{y}]/(\tilde{x}^2,\tilde{y}^2).$$
 \end{remark}
We are ready to state our first main result.
\begin{theorem}\label{maintheor}
    
    Let $(A, U)$ be an $H$-pair and $X$ be the corresponding hypersurface in~$\BP(A)$. Then there are finitely many $U$-orbits on $X$ if and only if the pair $(A, U)$ is isomorphic to one of the following pairs:
    \begin{align*}
    (\BK[x]/(x^{n+1}), \ U_i),\ &\text{where }\ U_i := \langle x, x^2,  \ldots, x^{i-1}, x^{i+1}, \ldots, x^n\rangle \ \text{ with } n-1 \leq i \leq n, \quad \text{or}\\
    (\BK[x,y]/(x^2, y^2), \ W),\ &\text{where }\ W = \langle x, y \rangle.
        \end{align*}
\end{theorem}
To prove Theorem \ref{maintheor}, we need the following lemma.
 \begin{lemma} \label{lem1}\begin{enumerate}
     \item  Let $(A, U)$ be an $H$-pair $(\BK[x]/(x^{n+1}), \ U_i),$ where $$U_i = \langle x, x^2 \ldots, x^{i-1}, x^{i+1}, \ldots, x^n\rangle$$ with $i > 1$. Consider the corresponding hypersurface $X$. Then there are finitely many $U$-orbits on~$X$ if and only if $ n-1\leq i \leq n$.
     \item Let  $(A, U)$ be an $H$-pair $(\BK[x,y]/(x^2, y^2), \ \langle x, y \rangle)$ and $X$ is the corresponding hypersurface. Then there are finitely many $U$-orbits on~$X$.
 \end{enumerate}
 
    \end{lemma}
    \begin{proof}
     First consider the case when an $H$-pair $(A, U)$ equals to $(\BK[x]/(x^{n+1}), \ U_i)$. By Theorem \ref{degree}, the degree of $X$ is equal to $i$. By Proposition~\ref{complement}, the complement to the open $U$-orbit in $X$ is the set 
        $$\{ \, p(z) \mid z \in \mf \text{ such that } z^i \in U\, \} = p(\mf^2).$$
    By Corollary \ref{morb}, for each point $p(z)$ from this set the $\mf$-orbit of $p(z)$ is contained in $X$. The total number of $\mf$-orbits on $\BP(A)$ is finite, see \cite[Proposition 3.7]{HT}. Each $\mf$-orbit either coincides with an $U$-orbit or is the union of infinite number of $U$-orbits. Therefore, the total number of $U$-orbits in $X$ is finite if and only if for all $z\in \mf^2\setminus \{0\}$ the $\mf$-orbit of $p(z)$ is equal to $U$-orbit of $p(z)$. By Proposition \ref{Uorbits} this is equivalent to 
    $$\Ann(z) + U = \mf, \ \forall z \in \mf^2.$$
    For $z\in \mf^2\setminus \mf^3$ we have $\Ann(z) = \mf^{n-1} = (x^{n-1})$ and $\Ann(z) \supseteq ( x^{n-1})$ for all other $z\in \mf^2$. Therefore, the total number of $U$-orbits in $X$ is finite  if and only if 
    $$(x^{n-1}) + U_i = \mf.$$
    It implies that $i = n$ or $n-1$.

    In the case when $(A, U) = (\BK[x,y]/(x^2, y^2), \ \langle x, y \rangle)$, the degree of $X$ is 2. Three $\mf$-orbits are contained in the complement to the open $U$-orbit in $X$. They are $p(x + \BK xy), p(y+ \BK xy)$ and $p(xy)$. It is easy to see that all these  $\mf$-orbits coincide with $U$-orbits.
    \end{proof}
\begin{proof}[Proof of Theorem \ref{maintheor}] 
    Let $(A, U)$ be an $H$-pair and suppose that corresponding hypersurface $X \subseteq \BP^n$ contains only a finite number of $U$-orbits. By Proposition \ref{algebras} and Remark \ref{rem} the algebra $A$ is isomorphic to $\BK[x]/(x^{n+1})$ or  $\BK[x,y]/(x^2, y^2)$.

    Consider the case $A \simeq \BK[x]/(x^{n+1})$. Let $U$ be an $(n-1)$-dimensional subspace in $\mf$, which generates $A$. Suppose that $\langle x^n \rangle \nsubseteq U$. Then 
    $$U = \langle x + \alpha_1x^n, x^2 + \alpha_2x^n, \ldots, x^{n-1} + \alpha_{n-1}x^n\rangle$$
    for some $\alpha_1,\ldots, \alpha_{n-1} \in \BK$. For all $\beta_2, \ldots, \beta_n \in \BK$ we consider an automorphism $\varphi$ of~$A$, $\varphi \colon x\mapsto x + \beta_2x^2 + \ldots +\beta_nx^n$. Then
\begin{align*}
  \varphi (x^k + \alpha_kx^n) =  (k\beta_{n-k+1} + h_k(\beta_2, \ldots, \beta_{n-k})+\alpha_k)x^n + s_k(x),
\end{align*}
where $h_k$ and $s_k$ are polynomials and the degree of $s_k$ is less than $n$.

    We take $\beta_{n-k+1} = -\frac{1}{k}(\alpha_{k} + h_{k}(\beta_2, \ldots, \beta_{n-k}))$ for all $k=1,\ldots,n-1$. Then 
    $$\varphi(x^k + \alpha_kx^n) \in \langle x, \ldots, x^{n-1}\rangle \quad \forall k=1,\ldots,n-1.$$
    Therefore, $\varphi(U) = \langle x, \ldots, x^{n-1} \rangle$.

    If $\langle x^n \rangle \subseteq U$ we can consider the canonical homomorphism $\pi\colon A \to A/\langle x^n \rangle \simeq \BK[x]/(x^n)$. Then $\pi(U)$ is an $(n-2)$-dimensional subspace that generates $A/\langle x^n\rangle$. Proceeding by induction we obtain that up to an automorphism of $A/\langle  x^n \rangle$
    $$\pi(U) = \langle x + \langle x^n \rangle, x^2 +  \langle x^n \rangle, \ldots, x^{i-1} + \langle x^n \rangle,  x^{i+1} + \langle x^n \rangle  , \ldots   \rangle$$
    for some $i \geq 2$. But then $U = U_i$.

    Now we consider the case $A \simeq \BK[x,y]/(x^2, y^2)$. If a 2-dimensional subspace $W$ in $\langle x, y, xy \rangle$ generates $A$ then $W = \langle x + \alpha xy, \ y + \beta xy\rangle$. Applying the automorphism of $A$
    $$x \mapsto x -\alpha xy,\quad y \mapsto y - \beta xy,$$
     we obtain that $W = \langle x, y\rangle$. Then Lemma \ref{lem1} completes the proof.
\end{proof}

By an $H$-pair we can find the equation of the corresponding hypersurface $X$. For example, we consider the $H$-pair $(A, U) = (\BK[x]/(x^3), \ \langle x \rangle)$. Then we apply Theorem \ref{equation}. If we choose a basis $1, x, x^2$ in $A$ then the map $\pi\colon A\to A/U$ can be given as follows:
$$z_0 + z_1x + z_2x^2 \mapsto z_0 + z_2x^2.$$
In this case, the degree of $X$ is 2. If we denote $z = z_1x + z_2x^2$ we obtain
$$\ln(1+\frac{z}{z_0}) = \frac{z}{z_0} - \frac{z^2}{2z_0^2} = \frac{z_1}{z_0}x + \frac{2z_0z_2 - z_1^2}{2z_0^2}x^2.$$
The hypersurface $X$ is then given by the following equation:
$$z_0^2 \cdot \pi(\ln(1+\frac{z}{z_0})) =  z_0z_2 - \frac{1}{2}z_1^2 = 0.$$
This is a non-degenerate quadric of rank 3.
Below we recall the definition of a non-degenerate hypersurface.

\begin{definition}\cite[Definition 2.22]{AZ}\label{nond}
    Suppose a projective hypersurface $X \subseteq \BP^n$ of degree $d$ is given by an equation $f(z_0, z_1,\ldots,z_n) = 0$. Then $X$ is called \emph{non-degenerate} if  there is no linear transformation of variables $z_0,\ldots, z_n$ that reduces the number of variables in $f$ to less than $n+1$.
\end{definition}
An $H$-pair $(A,U)$ defines a non-degenerate hypersurface if and only if~$\dim(\Soc(A)) = 1$ and $\mf = U \oplus \Soc(A)$, see \cite[Theorem 2.30]{AZ}. As a corollary we have the following result.
\begin{corollary}
    \label{finalcor}
Let $X\subseteq \BP^n$ be a projective hypersurface admitting an induced additive action with a finite number of orbits.
    \begin{enumerate}
        \item When $n=2$, $X$ is $\mathbb{P}^1$  embedded to $\BP^2$ via Veronese embedding of degree 2.
        \item When $n=3$, $X$ is one of the following projective surfaces: 
        \begin{itemize}
            \item[$(a)$] $\BP^1\times \BP^1$  embedded to $\BP^3$ as a non-degenerate quadric of rank 4 via Segre embedding;
            \item[$(b)$] The non-degenerate cubic
            $z_0^2z_3 - z_0z_1z_2 + \frac{z_1^3}{3} = 0$.
            \item[$(c)$] The degenerate quadric of rank 3.
        \end{itemize}
        \item When $n>3$, $X$ is either a non-degenerate hypersurface $X_n$ of degree $n$ or a degenerate hypersurface $Y_{n}$ of degree $n-1$. Moreover, $Y_n$ is a projective cone over $X_{n-1}.$
    \end{enumerate}

\end{corollary}

In Table \ref{tab:metka} one can find the equations of the resulting hypersurfaces of dimensions 1--4.
\bigskip
\begin{center}
\begin{table}[]
\begin{tabular}{|c||c|}
\hline \\[-1em]
 $\dim\ X$ & \ The equation of a hypersurface \\ 
 \hline \hline \\[-1em]
1 &\ \ \ $z_0z_2 - \frac{1}{2}z_1^2 = 0$\\ 
\hline \\[-1em]
2 &\ \ \ $z_0z_3 - z_1z_2 = 0$ \\
\hline \\[-1em]
2 &\ \ \ $z_0^2z_3 - z_0z_1z_2 + \frac{z_1^3}{3} = 0$ \\
\hline \\[-1em]
2 &\ \ \ $z_0z_2 - \frac{1}{2}z_1^2 = 0$ \\
\hline \\[-1em]
3 &\ \ \ $z_0^3z_4 - z_0^2z_1z_3 + \frac{z_0^2z_2^2}{2} + z_0z_1^2z_2 - \frac{z_1^4}{4} = 0$ \\
\hline \\[-1em]
3 &\ \ \ $z_0^2z_3 - z_0z_1z_2 + \frac{z_1^3}{3} = 0$ \\
\hline \\[-1em]
4 &\ \ \ $z_0^4z_5 - z_0^3z_1z_4 - z_0^3z_2z_3 + z_0^2z_1^2z_3 + z_0^2z_1z_2^2 - z_0z_1^3z_2 + \frac{z_1^5}{5} = 0$ \\
\hline \\[-1em]
4 &\ \ \ $z_0^3z_4 - z_0^2z_1z_3 + \frac{z_0^2z_2^2}{2} + z_0z_1^2z_2 - \frac{z_1^4}{4} = 0$ \\
\hline
\end{tabular}\vspace{0.3 cm}
\caption{}\label{tab:metka}
\end{table}
\end{center}
\smallskip

\begin{remark}
It was proven in \cite{ABZ} that for each $n \in \BN$ there is a unique hypersurface in $\BP^n$ of degree $n$ that admits an additive action. It is the hypersurface which corresponds to the $H$-pair $(\BK[x]/(x^{n+1}), \langle x, x^2, \ldots, x^{n-1}\rangle).$ So we see that except the quadric of rank 4 in $\BP^3$ this is also a unique non-degenerate hypersurface with an additive action with a finite number of orbits.
\end{remark}

\begin{remark}
    Using Theorem \ref{equation} it is easy to see that the degenerate hypersurface corresponding to $H$-pair $(\BK[x]/(x^{n+1}), \ U)$ with $U_ := \langle x, x^2 \ldots, x^{n-2}, x^{n}\rangle$ can be given by the same equation in $\BP^n$ as the non-degenerate hypersurface in $\BP^{n-1}$ corresponding to $H$-pair $(\BK[x]/(x^{n}), \ U)$ with $U := \langle x, x^2 \ldots, x^{n-2}\rangle$.
\end{remark}

\section{Properties of hypersurfaces}\label{PROP}

\subsection{Orbits} In this section we describe the structure of orbits on hypersurfaces that we found in the previous section.

Let $A$ be a local algebra of dimension $n+1$ and $\mf$ is the maximal ideal in $A$. Then the $\mf$-orbit in $\BP(A) = \BP^n$ of an element $p(z)$ for $z\in A$ is the set 
$$O_{\mf}(p(z)) = \{p(w)\mid w\  \text{is associated with } z \}.$$
Here by $O_\mf(y)$ for $y\in \BP^n$ we mean the $\mf$-orbit of $y$. Hence, for $A = \BK[x]/(x^{n+1})$ we have $n+1$ $\mf$-orbits:
$$O_{\mf}(p(1)), O_{\mf}(p(x)), \ldots, O_{\mf}(p(x^n)).$$

Now we consider the non-degenerate hypersurface $X$ corresponding to the $H$-pair $(\BK[x]/(x^{n+1}),U)$, where $U = \langle x, x^2, \ldots x^{n-1}\rangle.$ Since the number of $U$-orbits in $X$ is finite by Proposition \ref{Uorbits} all $U$-orbits in $X$ except $U$-orbit of $p(1)$ coincide with $\mf$-orbits.

Among points $p(x),\ldots, p(x^n)$ exactly points $p(x^2), \ldots, p(x^n)$ belong to $X$. So there are exactly $n$ $U$-orbits in $X$:
$$O_U(p(1)),O_U(p(x^2)),\ldots, O_U(p(x^n)).$$
Similarly, here we denote by $O_U(y)$ for $y \in X$ the $U$-orbit of $y$.

For $k\geq 2$ we have 

\begin{align*}
  \mathrm{dim}\ O_U(p(x^k)) = \dim O_\mf (p(x^k)) = \dim \mf - \dim \mathrm{St}_\mf (p(x^k)) =\\ =n - \dim(\mathrm{Ann}(p(x^k)) = n - \dim\langle x^{n-k+1},\ldots, x^n \rangle = n-k.
\end{align*}
At the same time, $O_U(p(x^k)) \subseteq p(\langle x^k,\ldots, x^n \rangle).$ The last set is irreducible, closed in $X$ and has the dimension $n-k$. So $\overline{O_U(p(x^k))} = p(\langle x^k,\ldots, x^n\rangle).$ It implies that $O_U(p(x^k)) \subseteq \overline{O_U(p(x^l))}$ if and only if $l \leq k.$

All these arguments also work for the case of the degenerate hypersurface corresponding to the $H$-pair $(\BK[x]/(x^{n+1}), \langle x, \ldots, x^{n-2}, x^n\rangle).$ So we obtain the following proposition.

\begin{proposition}
    Let $X$ be the hypersurface corresponding to the $H$-pair $(\BK[x]/(x^{n+1}), \langle x, \ldots, x^{n-1}\rangle)$ or $(\BK[x]/(x^{n+1}), \langle x, \ldots, x^{n-2}, x^n\rangle).$ Then there are exactly~$n$ $U$-orbits $O_0,\ldots, O_{n-1}$ in $X$ with $\dim O_i = i$ and $O_i \subseteq \overline{O_j}$ if and only if $j \geq i$.
\end{proposition}

It is easy to check that there are 4 orbits in $\BP^1\times \BP^1.$ One is two-dimensional and open, two orbits are 1-dimensional and one orbit is a fixed point which is contained in the closures of all others orbits.

\subsection{Smoothness and normality} It is clear that $\BP^1\times \BP^1$ is smooth and normal. By \cite[Proposition 4]{AP} a smooth hypersurface admitting an induced additive action is a non-degenerate quadric. So the only smooth hypersurfaces admitting an induced additive actions with a finite number of orbits are non-degenerate quadrics of rank 3 and 4 in $\BP^2$ and $\BP^3$ respectively. 

To study normality we use Proposition 3 in \cite{ABZ}. Let $X$ be a hypersurface admitting an additive action and $(A,U)$ is the corresponding $H$-pair. Then $X$ is given in $\BP^n$ by the equation 
$$z_0^d\pi(\ln(1+\frac{z}{z_0})) = \sum_{k=1}^d z_0^{d-k}f_k = 0,$$
where $f_k$ is a homogeneous polynomial of degree $k$ and $d$ is the degree of the hypersurface; see Theorem \ref{equation}. Let $f_d = p_1^{a_1}\ldots p_r^{a_r}$, where $p_1,\ldots, p_r$ are distinct coprime irreducible polynomials and $a_i > 0.$ Denote  $\tilde{f}_d = \frac{f_d}{p_1\ldots p_r}.$

\begin{proposition}\cite[Proposition 3]{ABZ}
The hypersurface $X$ is normal if and only if the polynomials $\tilde{f}_d$ and $f_{d-1}$ are coprime.
\end{proposition}

It implies that a hypersurface $\{F = 0\} \subseteq \BP^n$ admitting an induced additive action is normal if and only if the hypersurface $\{F = 0\} \subseteq \BP^{n+1}$ is normal. So it is enough to check normality for the hypersurfaces corresponding to $H$-pairs $(\BK[x]/(x^{n+1}), \langle x,\ldots, x^{n-1}\rangle)$. In this case $d = n$ and we have
$$\tilde{f_d} = (-1)^n\frac{z_1^{n-1}}{n},\ \text{and} \ f_{d-1} = (-1)^{n-1}z_1^{n-2}z_2.$$
Here $z_0, z_1\ldots, z_n$ are coordinate functions on $\BP^n$ corresponding to the basis $1,x,\ldots, x^n$ in $\BK[x]/(x^{n+1}).$ These polynomials are coprime if and only if $n = 2.$ So we obtain the following result.

\begin{proposition}
    \begin{enumerate}
        \item Let $X\subseteq \BP^n$ be a smooth hypersurface admitting an additive action with a finite number of orbits. Then either $n = 2$ and $X = \{z_2z_0 - \frac{1}{2}z_1^2 = 0\}$ or $n = 3$ and $X = \{z_0z_3 - z_1z_2 = 0 \}.$
        \item Let $X\subseteq \BP^n$ be a normal hypersurface admitting an additive action with a finite number of orbits. Then either $n = 2$ and $X = \{z_2z_0 - \frac{1}{2}z_1^2 = 0\}$ or $n = 3$ and $X = \{z_2z_0 - \frac{1}{2}z_1^2 = 0\}$ or $X = \{z_0z_3 - z_1z_2 = 0 \}$. 
     \end{enumerate}
\end{proposition}

\subsection{The number of additive actions} By \cite[Theorem 2.32]{AZ} a non-degenerate hypersurface in $\BP^n$ admits at most one induced additive action. At the same time, it was proven in \cite{Be} that if a degenerate hypersurface in $\BP^n$ admits an induced additive action then it admits at least two non-isomorphic induced additive actions. Here we will prove that if a degenerate hypersurface in $\BP^n$ admits an induced additive action with a finite number of orbits then it admits exactly two induced additive actions. 

Let $X = \{ F = 0\} \subseteq \BP(V) = \BP^n$ be a degenerate hypersurface, where $F$ is a homogeneous polynomial of degree $d$ and $V$ is a vector space. Suppose $(A,U)$ is the corresponding $H$-pair. Note that $A \simeq V$ as a vector space. The polynomial $F$ corresponds to a $d$-linear form $\overline{F}:V\times \ldots \times V \to \BK.$ We denote  $J = \mathrm{Ker}\ F.$ Then by \cite[Lemma 2.19]{AZ} $J$ is an ideal in $A$ contained in $U$ and $J$ is the unique maximal ideal among all ideals in $A$ contained in~$U$. Then $F$ defines a polynomial $\widetilde{F}$ on $\BP(V/J)$ and the hypersurface $\{\widetilde{F} = 0\} \subseteq \BP(V/J)$ is non-degenerate. Moreover, additive action on the hypersurface $\{F = 0\}$ induces an addititve action on $\{\widetilde{F} = 0\}$ which corresponds to the $H$-pair $(A/J, U/J);$ see \cite[Corollary 2.23]{AZ}. 

Now we assume that $(A,U) = (\BK[x]/(x^{n+1}), \langle x, x^2, \ldots, x^{n-2}, x^n\rangle)$. Then $J = (x^n)$ and $(A/J, U/J) = (\BK[x]/(x^n), \langle x, \ldots, x^{n-2}\rangle).$ The hypersurface corresponding to the $H$-pair $(A/J, U/J)$ is non-degenerate. So there is only one additive action on it.

Consider an $H$-pair $(B, W)$ which corresponds to an induced additive action on the hypersurface $\{F = 0\}.$ Then there is  one-dimensional ideal $P$ in $B$ with $P\subseteq W$ and $(B/P, W/P) = (\BK[x]/(x^n), \langle x, \ldots, x^{n-2}\rangle)$. Since there are no non-zero ideals in $\BK[x]/(x^n)$ contained in $\langle x, \ldots, x^{n-2}\rangle$ in this case $P$ will be automatically maximal among all ideals in $B$ contained in $W$.

Let $y$ be a preimage of $x$ in $W$ with respect to the canonical projection $B \to B/P$. Denote by $z$ a basis in $P$. Then $z\in \mathrm{Soc}(B).$

The image of $y^n $ in $B/P$ is zero  so $y^n \in \langle z \rangle.$ If $y^n \neq 0$ then up to a scalar $y^n = z.$ In this case $B = \BK[y]/(y^{n+1})$ and $W = \langle y, \ldots, y^{n-2}, y^n\rangle$. If $y^n = 0$ then $B = \BK[y,z]/(y^n, z^2, zy)$ and $W = \langle y, \ldots, y^{n-2}, z\rangle$. So we have proved the following proposition.

\begin{proposition}
    Let $X$ be a hypersurface in $\BP^n$ which corresponds to the $H$-pair $(\BK[x]/(x^{n+1}), \langle x, \ldots, x^{n-2}, x^n\rangle)$. Then there are exactly two induced additive actions on $X$. One corresponds to the $H$-pair $(\BK[x]/(x^{n+1}), \langle x, \ldots, x^{n-2}, x^n\rangle)$ and the other one corresponds to the $H$-pair $(\BK[y,z]/(y^n, z^2, zy),\langle y, \ldots, y^{n-2}, z\rangle)$. 
\end{proposition}

It is not difficult to describe the additive action corresponding to the $H$-pair $(\BK[y,z]/(y^n, z^2, zy),\langle y, \ldots, y^{n-2}, z\rangle)$. Let $X_{n-1}$ be the hypersurface in $\BP^{n-1}$ corresponding to the $H$-pair $(\BK[y]/(y^n), \langle y, y^2,\ldots, y^{n-2} \rangle)$. We denote by 
$z_0,\ldots, z_{n-1}$ the homogeneous coordinates on $\BP^{n-1} $ corresponding to the basis $1,y,\ldots, y^{n-1}$. Then the element 
$$(s_1, \ldots, s_{n-2}) =  s_1y + s_2y^2 + \ldots + s_{n-2}y^{n-2} \in U$$ 
acts on $X_{n-1}$ in the following way
$$(s_1,\ldots, s_{n-2})\circ [z_0:\ldots : z_{n-1}] = [z_0 : z_1 + g_1: \ldots :z_{n-1} + g_{n-1}],$$
where $g_1,\ldots, g_{n-1}$ are polynomials  in variables $z_0,\ldots, z_n, s_1,\ldots, s_{n-1}.$ This action is linear, so the polynomials $g_1, \ldots, g_{n-1}$ are homogenous polynomials of degree 1 in variables $z_0,\ldots, z_{n-1}$. Moreover, one can check that the polynomial $g_i$ does not depend on $z_i,\ldots, z_n.$

Suppose that $X_{n-1}$ is given in $\BP^{n-1}$ by a homogeneous polynomial $F$. Then the hypersurface $Y_n$ corresponding to the $H$-pair $(\BK[y,z]/(y^n, z^2, zy),\langle y, \ldots, y^{n-2}, z\rangle)$ is given in $\BP^n$ by the same polynomial $F$ and the corresponding additive action is given by the following formula:

$$(s_1,\ldots, s_{n-2}, s_{n-1})\circ [z_0:\ldots : z_{n-1}:z_n] = [z_0 : z_1 + g_1: \ldots :z_{n-1} + g_{n-1}:z_n + s_{n-1}z_0],$$
where $z_0,\ldots, z_n$ are homogeneous coordinates on $\BP^n$ corresponding to the basis $1, y, y^2,\ldots, y^{n-1}, z$ of $\BK[y,z]/(y^n,z^2,zy)$ and
$$(s_1,\ldots, s_{n-1}) = s_1y+s_2y^2 + \ldots + s_{n-2}y^{n-2} + s_{n-1}z \in U.$$

We denote by $O_i$ the orbit in $X_{n-1}$ of dimension $i$. Then the open orbit $O_{n-2}$ in $X_{n-1}$ is the set
$$O_{n-2} = \{[z_0:\ldots :z_{n-1}] \in \BP^{n-1} \mid F(z_0,\ldots, z_{n-1}) = 0\ \text{and} \ z_0 \neq 0\}.$$

Then the open orbit $O$ in $Y_n$ is the set
$$O = \{[z_0:\ldots :z_{n-1}:z_n] \in \BP^{n-1} \mid F(z_0,\ldots, z_{n-1}) = 0\ \text{and} \ z_0 \neq 0\}.$$

When $i<n-2$ the orbit $O_i$ in $X_{n-1}$ is the orbit of a point 
$P_i = [0: 0 : \ldots 1 : 0 : \ldots 0]$ where 1 stands at the coordinate $z_{n-i-1}.$ There are infinitely many orbits of dimension $i$ in $Y_{n}$. They are orbits of the points 
$$P_{i,c} = [0: 0 : \ldots 0: 1 : 0 : \ldots 0: c]$$ 
(again, 1  stands at the coordinate $z_{n-i-1}).$ If the closure $\overline{O_i}$ in $X_{n-1}$ is given by the set of homogeneous polynomials 
$$F_{i,1},\ldots, F_{i, r(i)}$$ then the closure of the orbit of $P_{i, c}$ in $Y_n$ is given by the set of polynomials 
$$F_{i,1},\ldots, F_{i, r(i)}, z_n - cz_{n-i-1}.$$
When $i>0$ the closure of the orbit of $P_{i,c}$ contains the orbit of $P_{i,c}$ and the orbits $O_0,\ldots, O_{i-1}.$ Here, we assume that $X_{n-1}$ is the subset of $Y_n$ which is given by the equation $z_n = 0.$

\subsection{Limit points of one-parameter subgroups}\label{LP} Let $X$
 be a complete variety and $\alpha:\mathbb{G}_a^n \times X \to X$ be an additive action on $X$. Let $O$ be the open orbit. We say that the additive action $\alpha$ satisfies \textbf{OP}-condition (one-parameter subgroups condition) if for every point $x\in X$ there is a point $y\in O$ and a one-dimensional subgroup $S\subseteq \mathbb{G}_a^n$ such that $x\in \overline{Sy}.$

 In \cite{CC} normal complete varieties with an additive action with a finite number of orbits and \textbf{OP}-condition were described. More precisely, the following holds. 

 \begin{theorem}\cite[Theorem A]{CC}
     Let $X$ be a complete variety and $\alpha: \mathbb{G}_a^n\times X \to X$ is an additive action. Then the following conditions are equivalent.
    \begin{enumerate}
        \item There are finitely many orbits on $X$ with respect to $\alpha$ and $\alpha$ satisfies \textbf{OP}-condition.
        \item The variety $X$ is a matroid Schubert variety and $\alpha$ is the corresponding additive action on~$X$.
     \end{enumerate}
 \end{theorem}

 One can find the definition of a matroid Schubert variety in the introduction of\cite{CC}. In this section we check what additive actions on projective hypersurfaces with a finite number of orbits satisfy \textbf{OP}-condition.

 Suppose $X$ is a complete variety and $\alpha:\mathbb{G}_a^n \times X\to X$ is an effective additive action on $X$ with an open orbit $O$. Then $O\simeq \mathbb{G}_a^n$ and we can assume that the group $\mathbb{G}_a^n$ is a subset in $X$. Then $\alpha$ satisfies \textbf{OP}-condition if and only if for every $\mathbb{G}_a^n$-orbit $O'$ there is a one-dimensional subgroup $S\subseteq \mathbb{G}_a^n$ such that $\overline{S}\cap O' \neq \emptyset$.

 Now we assume that $X$ is the projective hypersurface corresponding to the $H$-pair $(A, U) = (\BK[x]/(x^{n+1}), \langle x, \ldots, x^{n-1}\rangle)$ and $\alpha$ is the respective additive action. Consider a one-dimensional subgroup $S = \langle \alpha_1x + \ldots + \alpha_{n-1}x^{n-1}\rangle \subseteq U.$ The image of $S$ in $X$ is the subset $p(\mathrm{exp}(S)).$ Let $z_0,\ldots, z_n$ be the coordinates in $\BP(A) = \BP^{n}$ corresponding to the basis $1,x,\ldots, x^n$ of $A$. We have 
 $$\mathrm{exp}(S) = \{1 + t(\alpha_1x+\ldots +\alpha_{n-1}x^{n-1}) + \ldots + \frac{t^{n}\alpha_{1}^{n}x^{n}}{n!} \mid t \in \BK\}$$
 and $p(\mathrm{exp}(S))$ is a set of the form $\{[g_0(t): \ldots g_n(t)] \mid t\in \BK\}$, where $g_i$ is a  polynomial of degree $i$.

 If we consider a homogeneous polynomial $F(z_0,\ldots, z_n)$ which is equal to zero on $\overline{\exp(S)}$ then the polynomial $f(t) = F(g_0(t), \ldots, g_n(t))$ is zero. It implies that there are no monomials of the form $z_n^k$ in $F$. So $F(0,\ldots, 0, 1) = 0$. Therefore, $[0:\ldots : 1] \in \overline{\exp(S)}.$

 The set $\overline{\exp(S)}$ is the union of $\exp(S)$ and $[0:\ldots:0:1]$. Indeed, the group $S\simeq \BG_a$ acts on $\overline{\exp(S)}$ and $\exp(S)$ is an open orbit in $\overline{\exp(S)}$. So $\overline{\exp(S)} \setminus \exp(S)$ is a finite set of fixed points. But by \cite[Corollary 1]{BB} the set of $\BG_a$-fixed points on a complete variety is connected. So there is only one point in $\overline{\exp(S)}\setminus\exp(S).$ 

 Therefore, for any one-dimensional subgroup $S$ in $U$ the set $\overline{\exp(S)}$ intersects only two $U$-orbits: the open orbit and the fixed point $p(x^n).$ Hence, \textbf{OP}-condition holds only when $n=2.$ The same arguments show that \textbf{OP}-condition never holds for the $H$-pairs $(\BK[x]/(x^{n+1}), \langle x, \ldots, x^{n-2}, x^n\rangle)$ and it is easy to check that $\textbf{OP}$-condition holds for the additive action on $\BP^1\times \BP^1$. Thus, we obtain the following proposition.

 \begin{proposition}
     Let $X$ be a projective hypersurface which admits an additive action with a finite number of orbits and satisfying \textbf{OP}-condition. Then $X$ is either $\{z_2z_0 - \frac{1}{2}z_1^2 = 0 \}\subseteq \BP^2 $  or $\{z_0z_3 - z_1z_2 = 0\}\subseteq \BP^3.$ 
 \end{proposition}

 Since $\BP^1$ and $\BP^1\times \BP^1$ are normal varieties, we obtain the following corollary.

 \begin{corollary}
     Let $X$ be a matroid Schubert variety which can be embedded as a hypersurface in a projective space. Then $X$ is $\BP^1$ or $\BP^1\times \BP^1$.
 \end{corollary}

\end{document}